\documentclass[leqno,12pt]{article} 
\usepackage{amssymb, amsmath, amscd, latexsym, enumerate}
\usepackage[dvips]{graphicx} 
\usepackage[usenames]{color}
\usepackage{amsthm} 
\theoremstyle{plain} 
\newtheorem{theorem}{\indent\sc Theorem}[section]
\newtheorem{lemma}[theorem]{\indent\sc Lemma}
\newtheorem{corollary}[theorem]{\indent\sc Corollary}
\newtheorem{proposition}[theorem]{\indent\sc Proposition}

\newtheorem{fact}[theorem]{\indent\sc Fact}
\newtheorem*{theoremA}{\indent\sc Theorem~A}
\newtheorem*{theoremB}{\indent\sc Theorem~B}
\theoremstyle{definition} 
\newtheorem{definition}[theorem]{\indent\sc Definition}
\newtheorem{remark}[theorem]{\indent\sc Remark}

\makeatletter
\def\address#1#2{\begingroup
\noindent\parbox[t]{8.1cm}{%
\small{\scshape\ignorespaces#1}\par\vskip1ex
\noindent\small{\itshape E-mail address}%
\/: #2\par\vskip4ex}\hfill%
\endgroup}%
\makeatother

\newcommand{\Z}{\Bbb{Z}}
\newcommand{\C}{\Bbb{C}}
\newcommand{\D}{\Bbb{D}}
\newcommand{\R}{\Bbb{R}}
\newcommand{\N}{\Bbb{N}}
\renewcommand{\L}{\Bbb{L}}
\newcommand{\inner}[2]{\left\langle{#1},{#2}\right\rangle}
\renewcommand{\Re}{\operatorname{Re}}

\newcommand{\Sym}{\operatorname{Sym}}
\newcommand{\Id}{\operatorname{Id}}

\numberwithin{equation}{section}
\numberwithin{figure}{section}
\numberwithin{table}{section}
\title{\uppercase{Nonorientable maximal surfaces \\
                  in the Lorentz-Minkowski 3-space}}
\author{
\textsc{Shoichi Fujimori$^{*}$ and Francisco J. L\'{o}pez$^{\dagger}$} 
}
\date{} 

\begin{document}

\maketitle

\footnote{ 
$^{*}$Partially supported by JSPS Grant-in-Aid for 
  Young Scientists (Start-up) 19840035.
}
\footnote{
$^{\dagger}$Partially supported by MCYT-FEDER research project MTM2007-61775 and Junta de Andalucia Grant P06-FQM-01642.
}
%

\begin{abstract}
The geometry and topology of complete nonorientable maximal surfaces with lightlike singularities in 
the Lorentz-Minkowski 3-space are studied. Some topological congruence formulae for surfaces of this kind are obtained.
As a consequence, some existence and uniqueness results for maximal M\"{o}bius strips and maximal Klein bottles with one end 
 are proved.  
\end{abstract}

\section*{Introduction}
A maximal surface in the Lorentz-Minkowski 3-space $\L^3$ is  a spacelike
surface with zero mean curvature. Besides their mathematical
interest, these surfaces have a significant importance  in classical Relativity, Dynamic of Fluids, Cosmology, and so on 
(more information can be found for instance in \cite{mt, KI1, KI2}).  

Maximal surfaces in $\L^3$ share some properties with minimal surfaces in the
Euclidean 3-space $\R^3$. Both families  arise as solutions of variational problems: local maxima (minima)
for the area functional in the Lorentzian (Euclidean) case. Like minimal
surfaces in $\R^3,$ maximal surfaces in $\L^3$ also admit a Weierstrass
representation in terms of meromorphic data \cite{koba, koba1, Mc}.

Calabi \cite{ca}  proved that a complete  maximal surface in $\L^{3}$ is necessarily a
spacelike plane. Therefore, it is meaningless to consider global
problems on maximal and {\em everywhere  regular}  surfaces in $\L^{3}.$ However, physical and geometrical experience suggests to extend the global analysis to the wider family of complete maximal immersions with {\em singularities} (see \cite{KI1, KI2}). 
A point  of a maximal surface is said to be {\em singular} if the induced metric $ds^2$ degenerates at $p$. Throughout this paper, it will be always assumed that the complement of the singular set is a dense subset of the surface.  Roughly speaking, there are two kinds of singular 
points: classical branch points and lightlike singular points or points with lightlike tangent planes (see \cite{UY} for a good setting). Complete maximal surfaces with lightlike singularities and {\em no  branch points} have given rise to an interesting theory (see for instance \cite{FL, FLS, UY}). Following Umehara and Yamada \cite{UY},  this kind of surfaces  will be called (complete) {\em maxfaces}.  Generic singularities of maxfaces are classified in \cite{FSUY}.

Although the family of complete maxfaces is very vast, all previously known examples are orientable. Among them, we emphasize the Lorentzian catenoid  described by O. Kobaysshi \cite{koba1}, the Riemann type maximal examples exhibited by F. J. L\'{o}pez, R. L\'{o}pez and R. Souam \cite{LLS}, the high genus maxfaces produced by Umehara and Yamada \cite{UY}, the universal cover of  the entire maximal graphs with conical singularities described by Fernandez, Lopez and Souam  \cite{FL, FLS} and Kim-Yang  maximal examples  \cite{KY}.

The purpose of this paper is to study the geometry and topology of complete {\em nonorientable} maxfaces in $\L^3.$ It is interesting to notice that spacelike surfaces in $\L^3$ are orientable, and so the singular set of a nonorientable maxface is always non empty. We introduce the first basic examples of this kind of surfaces and obtain some natural characterization theorems. By definition, a nonorientable ``Riemann surface'' is a nonorientable surface endowed with an atlas whose transition maps are either holomorphic or antiholomorphic.

Like in the orientable case (see \cite{UY}),  a  conformal complete nonorientable maxface  $X:M \to \L^3$  is conformally equivalent to a compact nonorientable ``Riemann surface''  minus a finite set of points: $M=\overline{M}-\{p_1,\ldots,p_n\}.$ Furthermore, $dX$ has a ``meromorphic'' extension to $\overline{M}$    and the ends have finite total curvature. The  ``Gauss map'' $N$ of $M$ is well-defined on the complement of the singular set $S$ of $M,$ and takes values on $\Bbb{H}^2/\langle I \rangle,$ where $\Bbb{H}^2$ is the Lorentzian sphere of radius $-1$ and $I:\Bbb{H}^2\to\Bbb{H}^2$ is the antipodal map $I(p)=-p.$ Since $N$ is conformal, the composition $\hat{N}=p_s\circ N: M-S \to {\D}\equiv (\overline{\C}- \{|z|=1\})/\langle A \rangle$ is conformal as well, where $A$ is the complex involution $A(z)=1/\overline{z}$ and  $p_s$ is, up to passing to the quotients, the Lorentzian stereographic projection. Furthermore, $\hat{N}$ extends meromorphically to $\overline{M}$ and satisfies that $|\hat{N}(p_i)|\ne 1$ and $\hat{N}(S) \subset \{|z|=1\}.$ 

The immersion $X$ behaves like a spacelike sublinear multigraph around each end $p_i$ of $M,$ and labeling $\mu_i\geq 1$ as the winding number of $X$ at $p_i,$ the following  Jorge-Meeks type formula holds: 
$$\deg(\hat{N})=-\chi(\overline{M})+ \sum_{i=1}^{n} (\mu_i+1),$$ 
where $\deg(\hat{N})$ and $\chi(\overline{M})$ are the degree of $\hat{N}$ and  the Euler characteristic of $\overline{M},$ respectively (see \cite{Me, FL, FLS}).

The first part of the paper is devoted to prove the following topological congruence formulae:

\begin{theoremA}
If $X:M \to \L^3$ is a conformal complete nonorientable maxface with Gauss map $\hat{N},$ then
\begin{enumerate}
\item[$({\rm i})$] $\deg{\hat{N}}$ is even and greater than or equal to $4.$  
\item[$({\rm ii})$] If in addition $X$ has embedded ends $($that is to say, $\mu_i=1$ for all $i$$)$, then $\chi(\overline{M})$ is even.
\end{enumerate}
\end{theoremA}
In the second part, we produce the first known examples of complete nonorientable maxfaces. To be more precise, we describe the moduli space of complete maxfaces with the topology
 of a M\"{o}bius strip and Gauss map of degree four, and construct two complete one-ended Klein bottles, named $KB_1$ and $KB_2,$ with  Gauss map of degree four as well. Both $KB_1$ and $KB_2$ contain the $x_1$- and $x_2$-axes, and therefore their symmetry group contains four elements. Finally, we prove the following characterization theorem:

\begin{theoremB}
$KB_1$ and $KB_2$ are the unique complete maxfaces with the topology of a one-ended Klein bottle,  Gauss map of degree four and have at least four symmetries.
\end{theoremB}

The results in this work have been inspired by Meeks \cite{Me}, L\'{o}pez \cite{L1, L2} and L\'{o}pez-Mart\'{i}n  papers \cite{LM1, LM2} about complete nonorientable minimal surfaces in $\R^3.$


\section{Preliminaries}\label{sec:prelim}

Throughout this paper, we denote by $\overline{\C}$ the Riemann sphere.

Let $\L^3$ be the three dimensional Lorentz-Minkowski space with the 
metric $\inner{~}{~}=dx_1^2+dx_2^2-dx_3^2$. 
Let $M$ be a two dimensional manifold.  An immersion $X:M\to\L^3$ is 
called {\em spacelike} if the induced metric on the immersed surface is 
positive definite.  Using isothermal parameters, $M$ can be naturally 
considered as a Riemann surface and $X$ a conformal map. 
A conformal spacelike  immersion $X:M\to\L^3$ is said to be {\em maximal} if $X$ has vanishing mean curvature. 

Let $M$ be a Riemann surface, and let $X_1, X_2, X_3$ be three harmonic functions on $M$
 satisfying  that
$$dX_1^2+dX_2^2-dX_3^2= 0,$$
$$|dX_1|^2+|dX_2|^2+|dX_3|^2> 0.$$
Then the map $X:=(X_1, X_2, X_3):M\to\L^3$ gives 
a conformal maximal immersion with no branch points and eventually lightlike singularities (i.e., points where
the tangent plane is lightlike). The singularities correspond to the null set of $|dX_1|^2+|dX_2|^2-|dX_3|^2.$

If the nonsingular set $W=\{p \in M;\;\large(|dX_1|^2+|dX_2|^2-|dX_3|^2 \large)(p)> 0\}$ is dense in $M,$  $X$ is said to be a 
{\em maxface} \cite{UY}. 

We label $\phi_j$ as the holomorphic 1-form $dX_j$ ($j=1,2,3$), and call $g$ as the meromorphic function $i\phi_3/(\phi_1 - i\phi_2)$. Up  to a translation, 
\begin{equation}\label{eq:w-rep}
X=\Re\int (\phi_1, \phi_2, \phi_3), 
\end{equation}
where
\begin{equation}\label{eq:phi1phi2}
\phi_1 =\frac{i}{2}\left(\frac{1}{g}-g\right)\phi_3,\qquad
\phi_2 =\frac{1}{2}\left(\frac{1}{g}+g\right)\phi_3. 
\end{equation}
The induced metric $ds^2$ on $M$ (which is positive definite on $W$) is given by
\begin{equation}\label{eq:metric}
ds^2=|\phi_1|^2+|\phi_2|^2-|\phi_3|^2
    =\left(\frac{|\phi_3|}{2}\left(\frac{1}{|g|}-|g|\right)\right)^2.
\end{equation}
The singular set can be rewritten as $\{p\in M\,;\,|g(p)|=1 \}.$

\begin{remark}
Up to composing with the Lorentzian stereographic projection, $g$ coincides with the Gauss map of $X,$ and for this reason it
will be called as the meromorphic Gauss map of $X.$ For more details, see \cite{koba}.
\end{remark}

Conversely, let $M,g,\phi_3$ be a Riemann surface, a meromorphic function and  a holomorphic 1-form on $M,$ respectively,  satisfying that  the 1-forms 
$\phi_1$ and $\phi_2$ in equation (\ref{eq:phi1phi2}) are holomorphic, 
\begin{gather}
 |\phi_1|^2+|\phi_2|^2+|\phi_3|^2>0,  \quad \mbox{and}
\label{eq:R}
\\ 
\Re\int_\gamma (\phi_1,\phi_2,\phi_3)=(0,0,0) \quad 
\mbox{for all}\quad \gamma\in H_1(M,\Z).
\label{eq:P}
\end{gather} 
Then $X=\Re\int (\phi_1,\phi_2,\phi_3) : M \longrightarrow \L^3$ defines a maxface.

\begin{remark}
\begin{enumerate}
\item We call $(M,g,\phi_3)$ (or simply $(g,\phi_3)$) as the Weierstrass data of $X$. 
\item The condition \eqref{eq:R} is equivalent to 
     \begin{equation}\label{eq:liftmetric}
       \left(\frac{|\phi_3|}{2}\left(\frac{1}{|g|}+|g|\right)\right)^2>0,
     \end{equation}
      and simply means that $X$ has no branch points. 
\item The condition \eqref{eq:P} is the so called {\em period condition}, and guarantees that $X$ is well-defined on $M$. 
      This condition is equivalent to the following two equations:
     \begin{align}
       \int_\gamma g\phi_3 + \overline{\int_\gamma\frac{\phi_3}{g}} &= 0
       \quad \mbox{for all}\quad \gamma\in H_1(M,\Z),
       \label{eq:period12} \\
       \Re\int_\gamma \phi_3 &= 0
       \quad \mbox{for all}\quad \gamma\in H_1(M,\Z).
       \label{eq:period3}
     \end{align}
\item Since the coordinate functions of $X$ are harmonic, 
      the maximum principle implies that 
      there exist no compact maxfaces with empty boundary. 
\end{enumerate}
\end{remark}

The following notions  of completeness and finite type for maxfaces can be found in \cite{UY}. 

\begin{definition}
A maxface $X:M\to\L^3$ is said to be {\em complete} (resp. of {\em finite type}) if there exists a 
compact set $C$ and a symmetric (2,0)-tensor $T$ on $M$ such that $T$ 
vanishes on $M -  C$ and $ds^2+T$ is a complete (resp. finite 
total curvature) Riemannian metric. 
\end{definition}

\begin{proposition}[{\cite[Proposition 4.5]{UY}}] \label{pro:UY}
Let $X:M\to\L^3$ be a complete maxface. Then there exists a compact 
Riemann surface $\overline{M}$ and finite number of points 
$p_1,\ldots,p_n\in\overline{M}$ so that $M$ is biholomorphic to 
$\overline{M} - \{p_1,\ldots,p_n\}$. 
Moreover, the Weierstrass data $g$ and $\phi_3$ extend meromorphically 
to $\overline{M}$ and the limit normal vector at the ends is timelike. 
\end{proposition}
By definition, the genus of $X$ is the genus of $\overline{M}$. 
The removed points $p_1,\ldots,p_n\in\overline{M}$ correspond to the 
{\em ends} of $X$ (note that no end is accumulation point of the singular set).

\begin{theorem}[{\cite[Theorem 4.6]{UY}}]
Complete maxfaces are of finite type. 
\end{theorem}

It is not hard to see that any complete maxface $X:M \to \L^3$ is eventually a finite multigraph over any spacelike plane. Indeed, consider a spacelike plane  
$\Sigma\subset \L^3$ and let $p:\L^3 \to \Sigma$ denote the Lorentzian orthogonal projection on $\Sigma.$ Then take a solid circular cylinder $C \subset \L^3$ 
orthogonal to $\Sigma$ and containing all of the singularities of $X(M).$ By basic topological arguments  $X^{-1}(C)$ is compact, and it is not hard to check that the map $p \circ X:M-X^{-1}(C) \to \Sigma-C$ is a proper local diffeomorphism (and so a covering) with finitely many sheets, proving our assertion. The converse is also true (see \cite{FLS,FL} for more details).

Let $\mu_i$ denote the  winding number (or multiplicity) of the multigraph $X$ around $p_i.$ It is not hard to check that $\mu_i=\max\{\mbox{Ord}_{p_i} (\phi_j),\; j=1,2,3\}-1,$ where $\mbox{Ord}_{p_i}(\phi_j)$ is the pole order of $\phi_j$ at $p_i$ (see, for instance \cite{FL}). The following Jorge-Meeks type formula and Osserman-type inequality will be useful:

\begin{theorem}[\cite{FL, UY}]
\label{th:oss-ineq}
If  $X:\overline{M} - \{p_1,\ldots,p_n\}\to\L^3$ is a complete 
maxface with meromorphic Gauss map $g,$ then $$2\deg g= - \chi(\overline{M}) + \sum_{i=1}^n(\mu_i +1),$$ where $\chi(\overline{M})$ denotes the Euler characteristic of 
$\overline{M}$.   In particular,
\begin{equation}\label{eq:oss-ineq}
2\deg g \geq - \chi(\overline{M}) + 2n.
\end{equation}

Moreover, the equality holds if and only if $X$ is an embedding around any end of $M.$  
\end{theorem}

\section{Nonorientable maxfaces} 

Let $M'$ be a {\em nonorientable Riemann surface}, that is to say, a nonorientable surface endowed with
an atlas whose transition maps are holomorphic or antiholomorphic. Let $\pi: M\to M'$ denote the 
orientable conformal double cover of $M'$. 
\begin{definition}
 A conformal map $X':M'\to\L^3$ is said to be a {\em nonorientable maxface}  if the composition 
\[
X=X'\circ\pi : M\longrightarrow \L^3
\]
is a maxface. In addition,  $X'$ is said to be  complete if  $X$ is 
complete. 
\end{definition}

\begin{remark}
For any maxface $X:M\to\L^3$, 
regardless of whether $M$ is orientable or nonorientable, 
there exists a real analytic normal vector field which is well-defined on $M$. 
See Section 5 of \cite{KU} for more details.  
\end{remark}

Let $X':M'\to\L^3$ be a nonorientable maxface, and let $I:M \to M$ denote the antiholomorphic order two deck transformation 
associated to the orientable double cover $\pi :M\to M'.$ 
Since $X\circ I=X,$ then  $I^*(\phi_j)=\bar{\phi}_j$ ($j=1,2,3$), 
or equivalently, 
\begin{equation}\label{eq:gI-I(phi)}
g\circ I =\frac{1}{\bar{g}}
\quad\mbox{and}\quad
I^*(\phi_3)=\bar{\phi}_3.
\end{equation}
As a consequence, $I$ leaves invariant the singular set $\{p \in M\;;\; |g(p)|=1\}.$

Conversely, if $(g,\phi_3)$ is the Weierstrass data of a orientable 
maxface $X:M\to\L^3$ and $I$ is an antiholomorphic involution without 
fixed points in $M$ satisfying \eqref{eq:gI-I(phi)}, then the unique map $X':M'=M/\langle I\rangle\to\L^3$ satisfying that 
$X=X'\circ\pi$ is a nonorientable maxface. We call $(M,I,g,\phi_3)$ as the Weierstrass data of 
$X':M'\to\L^3$.

Assume that $X':M'=M/\langle I\rangle\to\L^3$ is complete. Then $I$ extends conformally to the compactification $\overline{M}$ of $M$ and $$M=\overline{M} - \{q_1,\ldots,q_m,I(q_1),\ldots,I(q_m)\},$$ where $q_1,\ldots,q_m \in \overline{M}.$ Consequently, $M'=\overline{M}' - \{\pi(q_1),\ldots,\pi(q_m)\},$ where $\overline{M}'=\overline{M}/\langle I \rangle$ is a compact 
nonorientable conformal surface of genus $2-\chi(\overline{M}')=2-(1/2) \chi(\overline{M}).$ By definition, the genus of $X'$ is the genus of $M'.$ 

\subsection{Topological congruence formulae for  nonorientable maxfaces}\label{sec:gaussmap}

Let $X':M'\to\L^3$ be a complete nonorientable maxface with   
Weierstrass data $(M,I,g,\phi_3),$ and label as $\pi:M \to M'$ as the orientable double cover of $M'.$ Denote by $A:\overline{\C} \to \overline{\C}$  the complex conjugation $A(z)=1/\overline{z},$ and consider the projection $p_0:\overline{\C} \to \overline{\D} \equiv \overline{\C}/\langle A \rangle.$   
\begin{definition}
The unique conformal map $\hat{g}:M' \to \overline{\C}/\langle A\rangle$ satisfying that  $\hat{g} \circ \pi=p_0 \circ g$ is said to be the Gauss map of $X'.$ 
\end{definition}

By Proposition \ref{pro:UY}, if $X'$ is complete then $\hat{g}$ extends conformally to the compatification $\overline{M}'$ of $M'.$ Moreover, $\hat{g}$ has the same degree as $g:\overline{M}\to \overline{\C}.$  The  Jorge-Meeks type formula in Theorem \ref{th:oss-ineq} gives 
$$\deg \hat{g} = - \chi(\overline{M}') + \sum_{i=1}^m(\mu_i+1),$$ where $\mu_i$ is the multiplicity of $X$ at $q_i,$ hence the inequality (\ref{eq:oss-ineq}) becomes: 
\begin{equation} \label{eq:oss-ineq-non}
\deg \hat{g} \geq - \chi(\overline{M}') + 2m,
\end{equation}
where $m$ is the number of ends of $M'.$

\begin{theorem}\label{th:deg(g)}
If $X'$ is complete then the degree of $\hat{g}$ is even.  
\end{theorem}
\begin{proof} Let $X':M'\to\L^3$ be a complete nonorientable maxface with the 
Weierstrass data $(M,I,g,\phi_3)$. 
As in the previous section, let $\overline{M}$ and $\overline{M}'$ be 
the compactifications of $M$ and $M'$, respectively.

Consider a meromorphic function $h$ on $\overline{M}$ such that $h\circ I=-1/\bar{h}$ (the existence of this kind of functions
is well known, see \cite{R}), and call $\hat{h} :\overline{M}'\to\R\Bbb{P}^2$ as the unique conformal map making the following
diagram commutative:
\[
\begin{CD}
\overline{M} @> h >> \overline{\C} \\
@V \pi VV  @VV \pi_0 V \\
\overline{M}' @> \hat{h} >> \R\Bbb{P}^2
\end{CD}
\]
Here $\R\Bbb{P}^2=\overline{\C}/I_0,$ where  $I_0(z)=-1/\bar z$ is the antipodal map, and $\pi_0:\overline{\C}\to\R\Bbb{P}^2=\overline{\C}/I_0$ 
is the natural projection. 
Since $\deg \pi = \deg \pi_0 =2$, the degree of $\hat{h}$ is  well-defined, and as a matter of fact $\deg \hat{h} = \deg h$.  

On the other hand, Meeks \cite[Theorem 1]{Me} proved the following fact:

\begin{fact}[{\cite[Theorem 1]{Me}}]\label{fc:Me}
Let $M_1$ and $M_2$ be compact surfaces without boundary and let 
$f:M_1\to M_2$ be a branched cover of $M_2$.  
If $\chi (M_2)$ is odd, then $\chi (M_1)$ and $\deg f$ are either both 
even or both odd.  If $\chi (M_2)$ is even, then $\chi (M_1)$ is even.  
\end{fact}
Therefore, we deduce that $\deg h = \deg \hat{h} \equiv \chi (\overline{M}') \pmod{2}.$ 

Up to composing $h$ with a suitable M\"{o}bius transformation of the form $L(z)=(z+a)/(\overline{a} z-1),$ we can suppose that $h(p) \neq 0,\infty$ for all
zero or pole $p$ of $g.$ Thus the meromorphic function $G:\overline{M}\to \overline{\C}$ defined by $G(z)=g(z)h(z)$ has 
\[
\deg G = \deg (gh) = \deg g + \deg h.
\]
Since $
G\circ I = (g\cdot h)\circ I = (g\circ I)(h\circ I)
         = \left(1/\bar g\right)\left(-1/\bar h\right) 
         = -1/\bar G,$ Meeks result gives that $\deg G \equiv \chi (\overline{M}') \pmod{2},$ and so $\deg(\hat{g})=\deg g\equiv 0 \pmod{2}$, proving the theorem. 
\end{proof}

\begin{corollary}\label{co:evengenus}
Let $X':M'\to\L^3$ be a complete nonorientable maxface with embedded ends.  
Then $X'$ has even genus.
\end{corollary}

\begin{proof}
Let $(M,I,g,\phi_3)$ be the Weierstrass data of $X':M'\to\L^3,$ and write $M=\overline{M} - \{q_1,\ldots ,q_m,I(q_1),\ldots ,I(q_m)\}.$  Since the ends are embedded,  Theorem~\ref{th:oss-ineq} gives that $2 \deg g = -\chi (\overline{M}) + 2\cdot (2m),$
hence $\chi (\overline{M})\equiv 0 \pmod{4}$ by Theorem~\ref{th:deg(g)}, which completes 
the proof.
\end{proof}

\begin{corollary}\label{co:deggeq4}
Let $X':M'\to\L^3$ be a complete nonorientable maxface.
Then the Gauss map of $X'$ has degree greater than or equal to $4.$
\end{corollary}
\begin{proof} Label $(M,I,g,\phi_3)$ as the Weierstrass data of $X'.$ 

If $X'$ has genus greater than  two, the corollary follows straightforwardly from equation (\ref{eq:oss-ineq-non}) and Theorem~\ref{th:deg(g)}.

Assume that $X'$ has genus two, and reasoning by contradiction suppose that  $\deg(\hat{g})=2.$  By  equation (\ref{eq:oss-ineq-non}) and Theorem~\ref{th:deg(g)}, $X'$ has an unique embedded end. Furthermore, up to Lorentzian isometries we may assume that $X'$ is asymptotic at infinity to either a horizontal plane or a horizontal upward half catenoid. In the first case, the third coordinate function of $X'$ is bounded, hence constant by the maximum principle (recall that the double cover $M$ is parabolic), which is absurd. In the second case, the third coordinate function of $X'$ has an interior minimum, contradicting the maximum principle for harmonic functions as well.

Finally, suppose that $X'$ has genus one, and as above suppose $\deg(\hat{g})=2.$ Up to a conformal transformation, we may assume that 
$M = \C - \{0\}$ and $I(z)=-1/\bar{z}$. Up to a suitable Lorentzian rotation, we will also assume $g(0)=0$ and $g(\infty)=\infty$.  Moreover, recall that $g$ and $\phi_3$ satisfy 
\eqref{eq:gI-I(phi)} and \eqref{eq:liftmetric} on $M.$  Since $g\circ I=1/\bar g,$  up to a suitable conformal transformation and rotation around the $x_3$-axis, we have that $g=z(z-r)/(r z+1),$ $r \in \R.$ By equation \eqref{eq:liftmetric} and the condition $I^*\phi_3=\bar{\phi}_3,$ we get that $\phi_3 = i s(rz+1)(z-r)z^{-2}dz,$ $s \in \R-\{0\}.$ 
A direct computation shows that \eqref{eq:period12} does not hold 
for a loop around $z=0,$ completing the proof.
\end{proof}

\begin{remark}
A similar result does not hold in the orientable case. The Lorentzian catenoid is a complete maxface of genus zero  and has degree one Gauss map. Moreover, there exist complete orientable one-ended genus one maxface with degree two Gauss map (see \cite{UY}), and complete orientable two-ended genus one maxface with degree two Gauss map (see \cite{KY}). 
\end{remark}

Theorem A in the introduction follows from Theorem \ref{th:deg(g)} and 
Corollaries \ref{co:evengenus} and \ref{co:deggeq4}.  

\section{Maximal M\"{o}bius strips with low degree Gauss map}
\label{sec:moebius}

This section is devoted to describe the family of one-ended genus one nonorientable complete maxfaces with degree four Gauss map.  

 Let $X':M' \to\L^3$ be a complete  maxface with the topological type of a M\"{o}bius strip. Without loss of generality we can write $M'=\R\Bbb{P}^2  -  \{\pi_0(0)\},$ where $\pi_0:\overline{\C}\to \R\Bbb{P}^2=\overline{\C}/\langle I_0 \rangle$ is the conformal universal cover and $I_0(z)=-1/\bar z.$   Call $(M=\C -  \{0\},I_0,g,\phi_3)$ as the Weierstrass data of $X',$ where $g$ is a meromorphic function of even degree (see Theorem \ref{th:deg(g)}). We are going to deal only with  the simplest case $\deg g=4.$ Up to  a suitable Lorentzian rotation, we will assume that $g(0)=0$ and $g(\infty)=\infty$.

\begin{lemma}\label{lm:branch(g)0}
In the above setting, the branching number of $g$ at $0$ and $\infty$ is even. 
\end{lemma}

\begin{proof}
Suppose that $g$ has a branch point of order three at $z=0$. After a rotation around the $x_3$-axis, we have that $g=z^4$ (recall that $g\circ I=1/\bar g$). 
Since $g$ has neither zeros nor poles on $M$, the same holds for $\phi_3$ by  \eqref{eq:liftmetric}. Taking into account that  $I^*\phi_3=\bar{\phi}_3$, we infer that $\phi_3=idz/z,$ contradicting that $\phi_3$ has no real periods on $\C-\{0\}.$  

Assume now that $g$ has a branch point of order one at $z=0$. In this case and after a rotation around the $x_3$-axis, we can put 
$$g=z^2\frac{(rz-1)(sz-1)}{(z+\bar{r})(z+\bar{s})}$$  
for some constants $r, s\in\C - \{0\}$, and  so by \eqref{eq:gI-I(phi)} and \eqref{eq:liftmetric} 
$$\phi_3 = i\frac{(rz-1)(z+\bar r)(sz-1)(z+\bar s)}{z^3}dz.$$ 
A direct computation shows that \eqref{eq:period12} does not hold 
for a loop around $z=0$, proving the Lemma.
\end{proof}

Suppose now that $g$ has a branch point of order two at $z=0$. Up to conformal transformations in $\C -  \{0\}$ and rotations around the $x_3$-axis, we may set $g=z^3(rz-1)/(z+r)$
for some real positive constant $r$.  Reasoning as in the proof of Corollary~\ref{co:deggeq4}, we get $\phi_3 = i(rz-1)(z+r)z^{-2}dz.$
Obviously $g\phi_3$ and $\phi_3/g$ have no residues at 
the ends, hence $\phi_1$ and $\phi_2$ have no real periods on $\C-\{0\}.$  Moreover, $\phi_3$ has no real periods if and only if $\int_\gamma \phi_3  = -2\pi (r^2-1)=0$
for any loop $\gamma$ winding once around $z=0,$ and so $r=1.$ 

Clearly $X$ is complete and its singular set is compact. 
Therefore, it induces a complete nonorientable maxface $X':\R\Bbb{P}^2  -  \{\pi(0)\} \to \L^3$. See the left-hand side of Figure~\ref{fg:max-moeb}. 

\begin{remark}
For each $k \in \N,$ the data $g=z^{2k+1}(z+1)/(z-1),$ $\phi_3 = i(z^2-1)z^{-2}dz$ on $\C - \{0\}$ determine a complete nonorientable maxface 
$X':\R\Bbb{P}^2  -  \{\pi_0(0)\}\to\L^3$ with $\deg g = 2k+2$.  
\end{remark}

\begin{remark}
If we set $g=z^2$ and $\phi_3 = i(z^2-1)z^{-2}dz$  on $\C -  \{0\},$ we obtain a Henneberg-type maximal immersion $X':\R\Bbb{P}^2  -  \{\pi_0(0)\}\to\L^3$ with 
singularities (see \cite{ACM}). This $X'$ is complete and has branch points at $z=\pm 1,$  so it is not a maxface. 
See Figure~\ref{fg:max-henn}. 
\end{remark}

\begin{figure}[htbp]
\begin{center}
 \includegraphics[width=.40\linewidth]{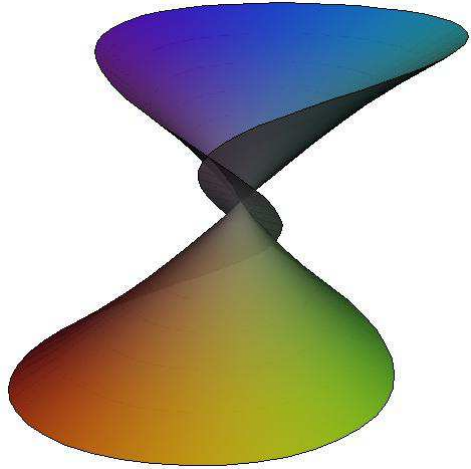} 
\end{center}
\caption{Henneberg-type maximal surface.}
\label{fg:max-henn}
\end{figure}

Assume now that $g$ has no branch points at the ends.  
As before, up to changes of coordinates and  rotations around the 
$x_3$-axis, we may set 
\[
g=z\frac{(rz-1)(sz-1)(tz-1)}{(z+r)(z+\bar{s})(z+\bar{t})}
\]
and 
\[
\phi_3 = i\frac{(rz-1)(z+r)(sz-1)(z+\bar s)(tz-1)(z+\bar t)}{z^4}dz 
\]
for some positive real constant $r$ and constants $s,t\in\C - \{0\}$. 
Take a loop $\gamma$ around $z=0$. 
Then direct calculation gives that
\begin{align*}
\int_\gamma g\phi_3+\overline{\int_\gamma \frac{\phi_3}{g}}
&= -4\pi \left(r^2+s^2+t^2+4rs+4st+4tr\right), \\
\frac{1}{2\pi} \int_\gamma \phi_3
&= (r^2-1)\left\{(|s|^2-1)(|t|^2-1)-s\bar t-\bar s t\right\} \\
& \qquad -r\left\{(|s|^2-1)(t+\bar t)+(|t|^2-1)(s+\bar s)\right\}.
\end{align*}
The arising moduli space of maxfaces is parameterized by the real analytic set of solutions of this system. For instance, the choice $r=1$, $s=e^{2\pi i/3}$ and $t=e^{-2\pi i/3}$ provides 
a surface in this family with high symmetry.  See the right-hand side of Figure~\ref{fg:max-moeb}.

\begin{figure}[htbp]
\begin{center}
\begin{tabular}{cc}
 \includegraphics[width=.40\linewidth]{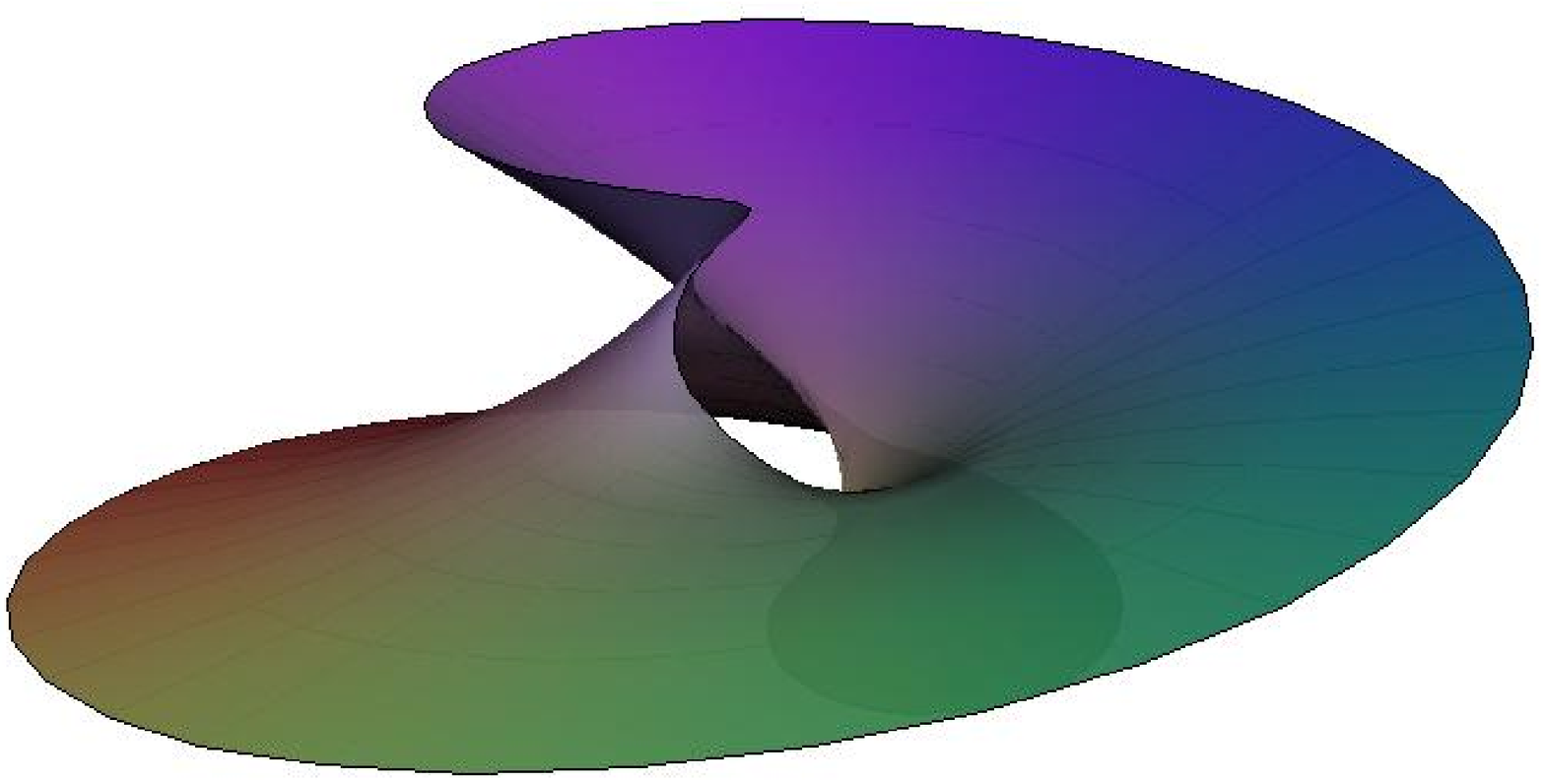} & 
 \includegraphics[width=.40\linewidth]{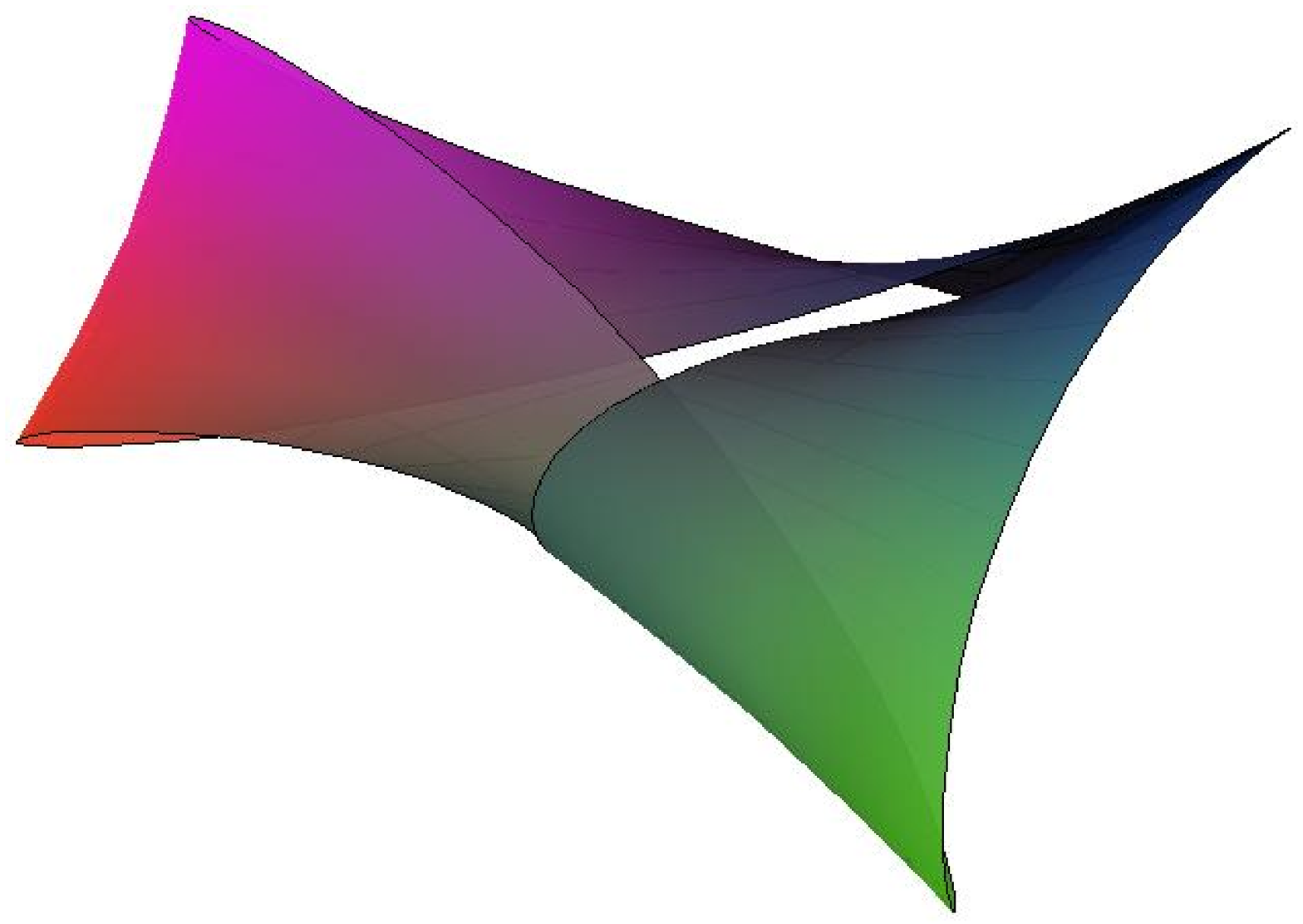} 
\end{tabular}
\end{center}
\caption{Maximal M\"{o}bius strips.  
         Left: $g$ has a branch point of order two at $z=0$.
         Right: $g$ has no branch points at the ends.}
\label{fg:max-moeb}
\end{figure}
\section{Maximal Klein bottles with one end}
\label{sec:klein}

In this section we construct complete maxfaces with the topology of a Klein bottle minus one point and the
 lowest Gauss map degree. Consider the genus one algebraic curve
\[
\overline{M}_r=\left\{(z,w_r)\in \overline{\C}^{2} \,;\, 
                      w_r^2=z\frac{rz-1}{z+r} \right\},
\quad
r\in \R-\{0\},
\]
and set 
$M_r=\overline{M}_r - \{(0,0),(\infty,\infty)\}$.  Define 
\[
I_{r}:\overline{M}_r\longrightarrow\overline{M}_r,\qquad
I_r(z,w_r)=\left(-\frac{1}{\bar z},-\frac{1}{\bar w_r}\right),
\]

\[
g_r=w_r\frac{z+1}{z-1},\qquad
\phi_3 = i\frac{z^2-1}{z^2}dz,
\]
and note that $I_{r}$ has no fixed points, and $g_r$ and $\phi_3$ satisfy \eqref{eq:liftmetric} and \eqref{eq:gI-I(phi)}.  
See Table~\ref{tb:g-phi3}.  
\begin{table} 
\begin{center}
\begin{tabular}{|c||c|c|c|c|c|c|}\hline 
$(z,w_r)$    & $(-r,\infty)$ & $(0,0)$ & $(r^{-1},0)$ & $(\infty,\infty)$ 
           & $(1,*)$ & $(-1,*)$ \\ \hline\hline
$g_r$        & $\infty^1$ & $0^1$ & $0^1$ & $\infty^1$
           & $\infty^1$ & $0^1$ \\ \hline 
$g_r\phi_3$  & --- & $\infty^2$ & $0^2$ & $\infty^4$
           & --- & $0^2$ \\ \hline 
$\phi_3$   & $0^1$ & $\infty^3$ & $0^1$ & $\infty^3$
           & $0^1$ & $0^1$ \\ \hline 
$\phi_3/g_r$ & $0^2$ & $\infty^4$ & --- & $\infty^2$
           & $0^2$ & --- \\ \hline 
\end{tabular}
\end{center}
\caption{The Divisors of the Weierstrass data.}
\label{tb:g-phi3}
\end{table} 
\begin{theorem}[Existence] \label{th:exist}
There are exactly two real values $r_1,$ $r_2 \in \R-\{0\}$ for which the 
maxface 
$$
X_r:M_r\ni p \mapsto 
    \Re\int^p \left(\frac{i}{2}\left(\frac{1}{g_r}-g_r\right),\,
                  \frac{1}{2}\left(\frac{1}{g_r}+g_r\right),\,1\right) \phi_3 
    \in \L^3
$$ 
is well-defined and induces a one-ended maximal Klein bottle $X_r':M_r/\langle I_{r} \rangle \to \L^3.$ 

Furthermore, the maxfaces $X_{r_1}'$ and $X_{r_2}'$ have Gauss map of degree four and four symmetries.
\end{theorem}
\begin{proof}
In order to solve the arising period problem, we first observe that  $\phi_3 = d(i(z^2+1)/z)$ is exact and \eqref{eq:period3} is satisfied.  
Moreover, $\phi_{1,r}=(i/2) (1/g_r-g_r) \phi_3$ and $\phi_{2,r}=(1/2) (1/g_r+g_r) \phi_3$ have no residues at the 
ends, hence it remains to check  \eqref{eq:period12} for 
$\gamma\in H_1(\overline{M}_r,\Z)$.  Let $c_1$ and $c_2$ be two loops in $\C-\{0,-r,1/r\}$ winding once around $[-r,0]$ and $[0,r^{-1}]$, 
respectively, and call $\gamma_1$ and $\gamma_2$ as their corresponding liftings  via $z$ to $\overline{M}_r$  (see Figure~\ref{fg:loop}).
\begin{figure}[htbp] 
\begin{center}
 \includegraphics[width=.70\linewidth]{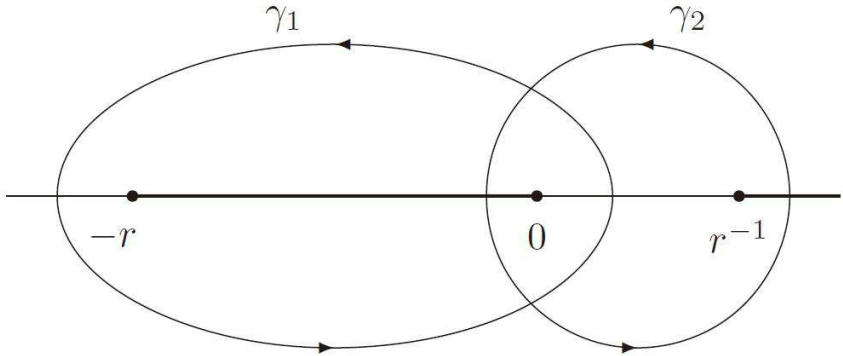}
\end{center}
\caption{Projection to the $z$-plane of the loops 
         $\gamma_1$ and $\gamma_2$.}
\label{fg:loop}
\end{figure} 

Let $(I_r)_*: H_1(\overline{M}_r,\Z) \to H_1(\overline{M}_r,\Z)$ denote the group isomorphism induced by $I_r.$ 
A straightforward computation gives that

\begin{equation}\label{lm:I*gamma}
(I_r)_*(\gamma_1)=-\gamma_1 \quad \mbox{and}\quad 
(I_r)_*(\gamma_2)= \gamma_2.  
\end{equation}

For any $j,k\in \{1,2\}$, we have
\[
 \int_{\gamma_j}\phi_{k,r}
=\int_{(I_r)_*(\gamma_j)}I_r^*(\phi_{k,r})
=\int_{(I_r)_*(\gamma_j)}\overline{\phi_{k,r}}
\]
and so $$\int_{\gamma_j}\phi_{k,r} + \int_{\gamma_j}\overline{\phi_{k,r}}
=\int_{(I_r)_*(\gamma_j)}\overline{\phi_{k,r}} + \int_{\gamma_j}\overline{\phi_{k,r}}.$$ 
Thus $$2\Re \int_{\gamma_j}\phi_{k,r} 
=\int_{\gamma_j+(I_r)_*(\gamma_j)}\overline{\phi_{k,r}}
=\int_{\gamma_j+(I_r)_*(\gamma_j)}\phi_{k,r},$$ and  $X_{r}=\Re\int (\phi_{1,r},\phi_{2,r},\phi_3) : M_r\longrightarrow\L^3$
is well-defined on $M_r$ if and only if 
\begin{equation}\label{eq:gamma+I*gamma=0}
\int_{\gamma_j+(I_r)_*(\gamma_j)}\phi_{k,r} = 0
\end{equation}
for all $j,k\in \{1,2\}$.  

\begin{lemma}\label{lm:gphi=0}
$X_{r}:M_r\to\L^3$ is well-defined on $M_r$ if and only if
\begin{equation}\label{eq:gphi=0}
\int_{\gamma_2}\frac{w_r(z+1)^2}{z^2}dz = 0.
\end{equation}
\end{lemma}

\begin{proof}
By (\ref{lm:I*gamma}) and 
\eqref{eq:gamma+I*gamma=0}, $X_{r}$ is well-defined if and only if 
\[
\int_{\gamma_2+(I_r)_*(\gamma_2)}\phi_{k,r} = 0
\]
holds for $k=1,2$.  
In other words, $X_{r}$ is well-defined if and only if  
\[
   \int_{\gamma_2}\left(\frac{1}{g_r}+g_r\right)\phi_3
 = \int_{\gamma_2}\left(\frac{1}{g_r}-g_r\right)\phi_3
 = 0
\]
holds, that is to say,
\[
\int_{\gamma_2}\frac{\phi_3}{g_r} = \int_{\gamma_2}g_r\phi_3 = 0
\]
holds.  
However, 
\[
   \int_{\gamma_2}\frac{\phi_3}{g_r}
 = \int_{(I_r)_*(\gamma_2)}I_r^*\left(\frac{\phi_3}{g_r}\right)
 = \int_{\gamma_2}\overline{g_r\phi_3}, 
\]
hence $X_{r}$ is well-defined on $M_r$ if and only if 
\[
\int_{\gamma_2}g_r\phi_3 = \int_{\gamma_2}\frac{w_r(z+1)^2}{z^2}dz = 0.
\qedhere
\]
\end{proof}
The period problem is equivalent to solve \eqref{eq:gphi=0}.  
To avoid divergent integrals we add the exact one-form $dF,$ where 
$$F= \frac{2w_{r}(z-2r^3z^2+r^2z(1+2z)-r(-1+2z+z^2))}{rz},$$ 
getting
\[
 \frac{w_r(z+1)^2}{z^2}dz +d F=-\frac{2w_{r} (-1+z+r(2-3z+r(-4+4r+3z)))}{r+z}dz.
\]

Since the right-hand side is a holomorphic differential on 
$M_r - \{(-r,\infty)\}$, the loop $\gamma_2$ can be collapsed over 
the interval $[0,r^{-1}]$ by Stokes theorem and $X_r$ is well-defined if and only if 
\[
h(r):=\int_0^{r^{-1}}
      -\frac{2|w_r(z)| (-1+z+r(2-3z+r(-4+4r+3z)))}{r+z}dz
     =0.
\]
A straightforward computation gives that $$h_+(0):=\lim_{r \to 0, \;r>0} h(r)=-\infty,\; h(+\infty):=\lim_{r \to +\infty} h(r)=-\pi,$$
$$h_-(0):=\lim_{r \to 0, \;r<0} h(r)=+\infty,\; h(-\infty):=\lim_{r \to -\infty} h(r)=+\pi.$$

Moreover, 
$$h(1/2)=\int_0^{2} \frac{2 |w_{1/2}(z)| (2-z)}{1+2z}dz> 0 
  \quad \mbox{and} \quad 
  h(1)=-\frac{4 \Gamma(3/4)^2 + \Gamma(-3/4) \Gamma(5/4)}{\sqrt{2 \pi}}<0,
$$ 
where $\Gamma$ is the classical
Gamma function. As a consequence,  $h$ has at least two roots in $(0,1)$ (and $X_r$ is well-defined at least for these two real values). 

Let us show that $h$ has exactly two real roots on $\R-\{0,1\}$ (recall that $h(1)<0$). 

It is clear that 
$$h'(r)=\frac{1}{2}\int_{\gamma_2} 
  \frac{\partial}{\partial r}\left( \frac{w_r(z+1)^2}{z^2}\right) dz,$$ 
hence a direct computation gives that 

\begin{equation}\label{eq:h'(r)}
h'(r)=\int_{0}^{r^{-1}} \frac{|w_r(z)| (1 + z)^2 (1 + z^2)}{2 z^2 (r + z) (-1 + r z)}dz.
\end{equation} 
Moreover,
$$\frac{w_r (1 + z)^2 (1 + z^2)}{2 z^2 (r + z) (-1 + r z)}dz+dH=-\frac{2 w_r (-r + 4 r^2 - z + 3 r z)}{r (r + z)}dz,$$ 
where 
$$
H=-\frac{w_r (r + 2 z - 2 r z - r z^2 + 4 r^2 z^2)}{r^2 z}. 
$$ 
Integrating by parts, we deduce that
$$h'(r)=\int_{0}^{r^{-1}} -\frac{2 |w_r(z)| (-r + 4 r^2 - z + 3 r z)}{r (r + z)}dz.$$

Now we rewrite $h(r)$ and $h'(r)$ as follows:
\begin{align*}
h(r) &=-2\left((3r^2-3r+1)A_1(r)+(r-1)(r^2+1)A_2(r)\right), \\
h'(r)&=-2\left(\frac{3r-1}{r}A_1(r)+rA_2(r)\right),
\end{align*}
where $A_i:\R -\{0\}\to\R_{+}$ ($i=1,2$) are the positive functions given by
\[
A_1(r)=\int_0^{r^{-1}}|w_r(z)|dz\quad\text{and}\quad
A_2(r)=\int_0^{r^{-1}}\frac{|w_r(z)|}{z+r}dz.
\]
If $h(r_0)=0$, then 
\[
A_2(r_0)=-\frac{3r_0^2-3r_0+1}{(r_0-1)(r_0^2+1)}A_1(r_0),
\]
hence necessarily $r_0<1$. Therefore $h(r_0)=0$ implies that
\[
h'(r_0)=-2\left(\frac{3r_0-1}{r_0}-\frac{r(3r_0^2-3r_0+1)}{(r_0-1)(r_0^2+1)}\right)A_1(r_0) 
=q(r_0)\int_0^{r_0^{-1}}|w_{r_0}(z)|dz,
\]
where $q:\R-\{0,1\} \to \R$ is the rational function
\[
q(r)=\frac{2(r^3-3r^2+4r-1)}{r(r-1)(r^2+1)}. 
\]
Basic algebra says that 
$$s=1 - \left( \frac{2}{3 (-9 + \sqrt{93})} \right)^{1/3} + \left( \frac{-9 + \sqrt{93}}{18} \right)^{1/3} \approx 0.317672
$$ 
is the unique real root of $q$ in $\R-\{0,1\},$ and an elementary analysis says that $q|_{(-\infty,0)}<0,$ $q|_{(0,s)}>0$ and $ q|_{(s,1)}<0.$

Assume for a moment that $h$ has a root in $(-\infty,0).$ Since $h_-(0)=+\infty$ and $h(-\infty)>0,$ we can find $s_0 \in (-\infty,0)$ such that $h(s_0)=0$ and $h'(s_0) \geq 0,$ contradicting that $q(s_0)<0.$ Therefore, the roots of $h$ (at least two) lie in $A=(0,1).$  Suppose that $h$ has three real roots on $A,$  and label $r_1<r_2<r_3$ as the  three smallest real roots of $h$ in $A.$ 

Since $h_+(0)=-\infty,$ $h$ must be increasing on $(r_1-\epsilon,r_1)$ for small $\epsilon$ and $h'(r_1) \geq 0.$ This implies that $r_1 \leq s.$ 

Let us show that $r_2 \geq s.$ If $r_1=s$ then $r_2>s$ and we are done. Suppose $r_1<s.$ In this case $h'(r_1) > 0$ and $h$ must be positive in $(r_1,r_2),$ hence $h$ must be decreasing on $(r_2-\epsilon,r_2)$ for small $\epsilon$ and $h'(r_2) \leq 0.$ This clearly implies that $r_2 \geq s.$ 

As a consequence,  $r_3>s$ and $h'(r_3)<0,$ which obviously contradicts that $h$ increasing on $(r_3-\epsilon,r_3)$ for small $\epsilon$ and proves our assertion.

This proves that $h$ has exactly two real roots $r_1$ and $r_2$ lying in $(0,1).$

Finally, observe that the transformations $T_0(z,w_r)=(z,-w_r),$ $T_1(z,w_r)=(\overline{z},\overline{w_r})$ and $T_2=T_1 \circ T_0$ on $\overline{M}_r$ induce the $180^\circ$-rotations about the $x_3$, $x_1$ and $x_2$ axes, respectively.
This implies that the maxface $X_r$ has four symmetries. 
\end{proof}

The values $r_1$ and $r_2$  can be estimated  using the Mathematica software, obtaining that $r_1\approx 0.17137$ and $r_2\approx 0.691724$.
See Figure~\ref{fg:h1}.

\begin{figure}[htbp] 
\begin{center}
\begin{tabular}{cc}
 \includegraphics[width=.40\linewidth]{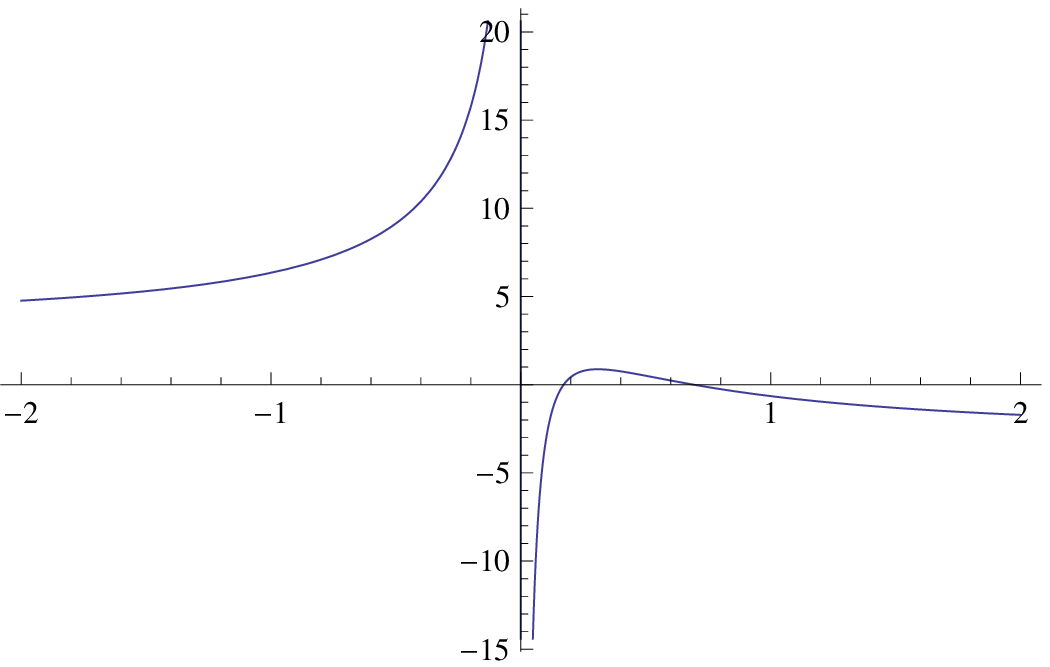} & 
 \includegraphics[width=.40\linewidth]{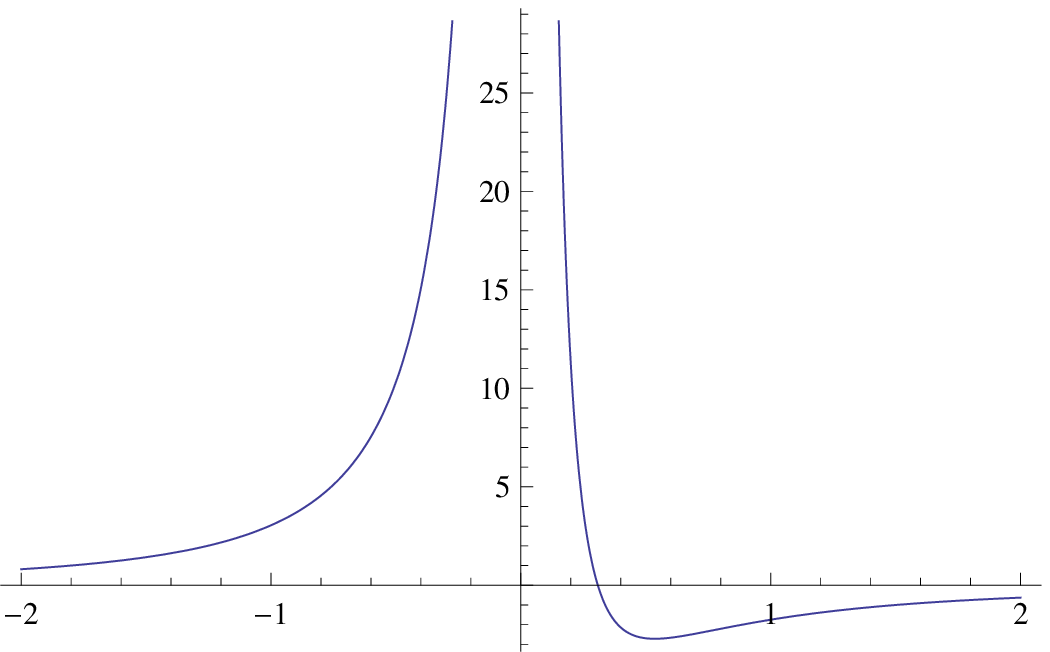} 
\end{tabular}
\end{center}
\caption{Left: The period function $h(r)$.  $h(r)=0$ when 
         $r\approx 0.17137$ and $r\approx 0.691724$. 
         Right: The derivative $h'(r)$ of $h(r)$.}
\label{fg:h1}
\end{figure} 

\begin{remark}
The above argument  is based on the construction of the 
L\'{o}pez' minimal Klein bottle \cite{L1}. 
The most significant difference is that in the Riemannian case the period problem has a unique solution.
\end{remark}
\begin{figure}[htbp] 
\begin{center}
\begin{tabular}{cc}
 \includegraphics[width=.33\linewidth]{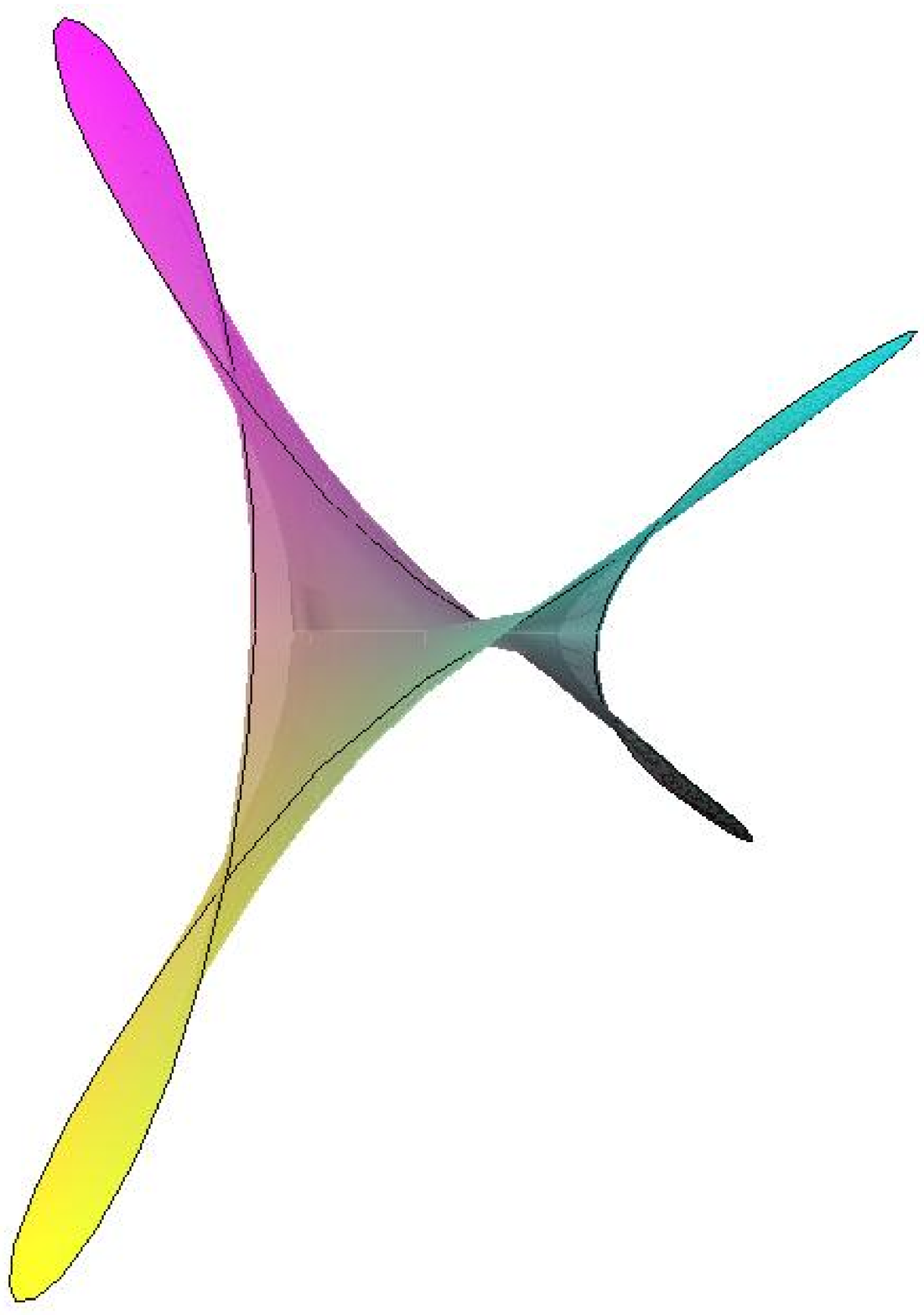} & 
 \includegraphics[width=.43\linewidth]{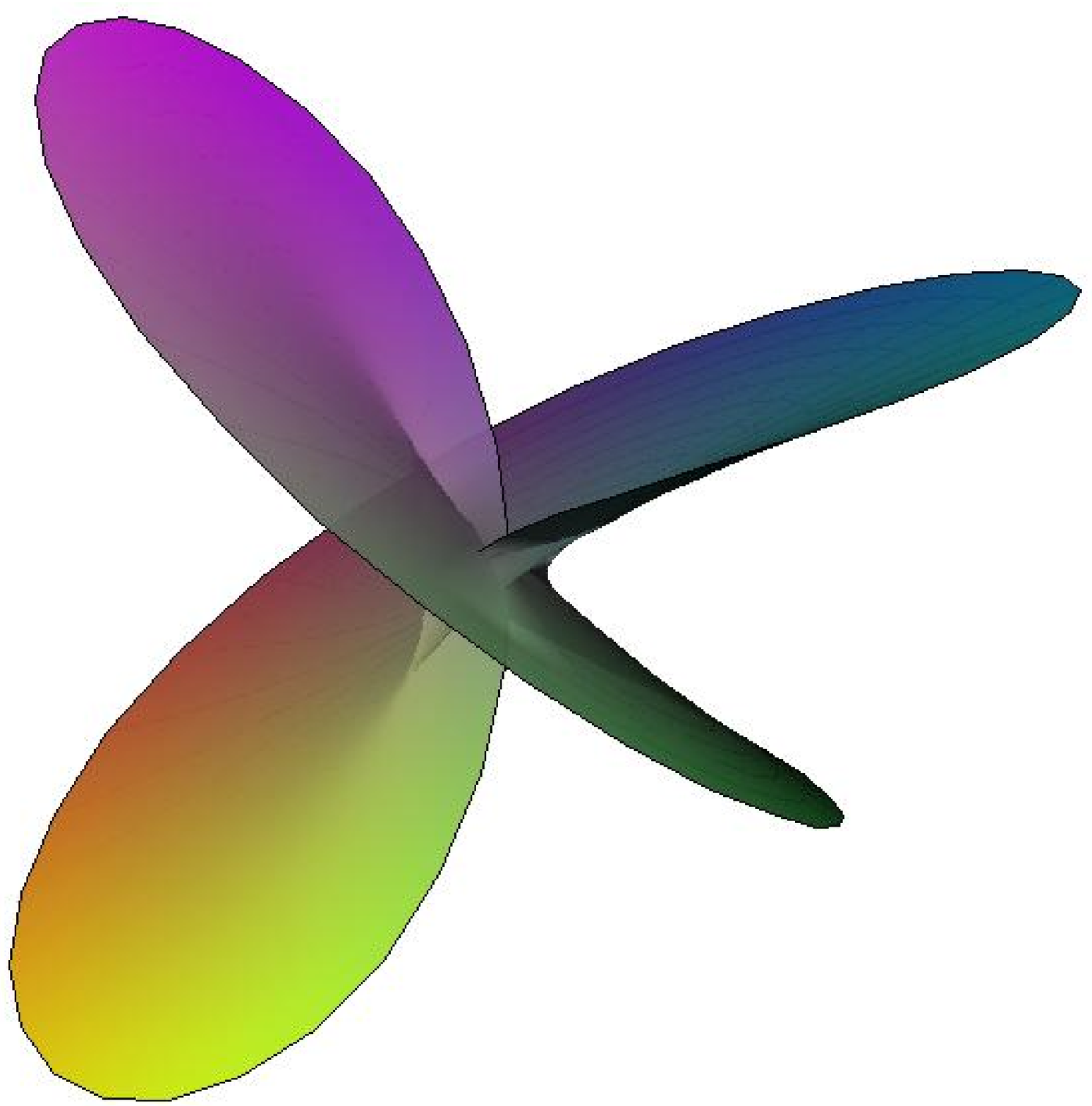} 
\end{tabular}
\end{center}
\caption{Maximal Klein Bottles with one end.  
         $r\approx 0.17137$ in the left, and $r\approx 0.691724$ 
         in the right.}
\label{fg:max-klein1}
\end{figure} 
%


The maximal Klein bottles exhibited in Theorem \ref{th:exist} can be characterized in terms of their symmetry:

\begin{theorem}[Uniqueness] \label{th:uniq}
Let $X':M'\to\L^3$ be a complete nonorientable maxface with genus two, 
one end and Gauss map of degree four.  
Assume that $X'$ has at least four symmetries.  
Then $X'$ is one of the examples constructed in Theorem \ref{th:exist}.  
\end{theorem}

\begin{proof}
By definition,  an intrinsic isometry $S:M'\to M'$ is said to be a symmetry of $X'$ if there exists 
 a Lorentzian isometry $\widetilde{S}:\L^3 \to \L^3$ such that $X \circ S=\widetilde{S} \circ X.$
Symmetries of $X'$ are conformal transformations and extend conformally to the compactification $\overline{M}'$ of $M'.$
We call $\Sym(X')$ as the symmetry group of $X'.$ 

Let $(M,I,g,\phi_3)$ denote the Weierstrass data of $X':M'\to\L^3$, and up to a Lorentzian isometry, suppose that $g(P)=1/g(I(P))=0.$ 
We know that $M=\overline{M} - \{P,I(P)\},$ where $\overline{M}$ is a conformal torus  
and $P\in\overline{M}$. As usual, label $\pi:\overline{M} \to \overline{M}'$ as the two sheeted
orientable double cover of  $\overline{M}'$ and $X=X' \circ \pi:M \to \L^3$ as the associated orientable maxface. For each $S \in \Sym(X'),$ let $\hat{S}:\overline{M} \to \overline{M}$
denote the unique holomorphic lifting of $S,$ that is to say, the unique orientation preserving transformation in $\overline{M}$
satisfying that $\pi\circ \hat{S}=S \circ \pi.$  Obviously $\hat{S} \circ I=I \circ \hat{S}.$ Write $\Sym_+(X)=\{\hat{S};\; S \in \Sym(X')\}$ and observe that 
$\Sym_+(X)$ is a group isomorphic to $\Sym(X').$
Note that $\hat{S}\in\Sym_+(X)$ satisfies $\hat{S}(P)=P$ or $\hat{S}(P)=I(P)$. 

Take an arbitrary $S\in \Sym(X'),$ and let us show that $S^2=\Id.$ 

Indeed, since $\{S^m;\; m \in \Z\}$ is a discrete group, there is $n \in \N$ such that $S^n=\Id$ and $S^j \neq \Id,$ $j=1,\ldots,n-1.$ 
Consider the orbit space $\overline{M}'/\langle S\rangle$ and the projection $\sigma:\overline{M}' \to \overline{M}'/\langle S\rangle.$ 
By Riemann-Hurwitz formula, $0=\chi(\overline{M}')=n \chi(\overline{M}'/\langle S\rangle) -V_S,$ where $V_S$ is the total branching number of $\sigma.$  
Since $S(\pi(P))=\pi(P)$, we get $V_S \geq n-1$ and $0\leq n \chi(\overline{M'}/\langle S\rangle) -n+1.$ This implies that  $\chi(\overline{M'}/\langle S\rangle)=1$ and $V_S=n.$ Therefore, there exists $Q \in \overline{M'}$ and a divisor $k$ of $n$  such that $n-k=1.$ This is only possible when $n=k+1=2,$ proving our assertion.

As a consequence, $T^2=\Id$ for all $T\in \Sym_+(X).$  Moreover, up to a rotation about the $x_3$-axis,  $g \circ T \in \{\pm g, 1/g\}$  and $T^*(\phi_3)=\pm \phi_3$ for any $T\in \Sym_+(X).$ To check this, just take into account that $g \circ T=L \circ g,$   where $L$ is the M\"{o}bius transformation induced by the linear part of $T$ (here we are identifying $\overline{\C}-\{|z|=1\}$ with the Lorentzian sphere of radius $-1$ via the Lorentzian stereographic projection). The normalization $g(P)=1/g(I(P))=0$ and the fact $T^2=\Id$ show that $g \circ T \in \{\pm g, \theta/g\}$, $|\theta|=1,$ and so the desired statement.

Let us show that there exists $T_0 \in \Sym_+(X),$ $T_0 \neq \Id,$ satisfying that $T_0(P)=P.$ Indeed, since $\#\Sym_+(X)\geq 4$, we can find $T_1,$ $T_2 \in \Sym_+(X)-\{\Id\}$ with $T_1 \neq T_2.$ If $T_1(P)=T_2(P)=I(P)$ (otherwise we are done), it suffices to take $T_0=T_1 \circ T_2.$  

Consider a such $T_0,$ and note that $T_0(I(P))=I(P)$ as well, that is to say, $T_0$ has at least two fixed points. By the Riemann-Hurwitz formula
$$0=\chi(\overline{M})=2 \chi(\overline{M}/\langle T_0\rangle) -V\geq 2 (\chi(\overline{M}/\langle T_0\rangle) -1),$$ where $V$ is the number of fixed points of $T_0.$ This clearly implies that $\chi(\overline{M}/\langle T_0\rangle)=2$ and $V=4.$ In other words, $\chi(\overline{M}/\langle T_0\rangle)=\overline{\C}$ and $T_0$ has in fact four fixed points, namely $\{P,I(P),Q,I(Q)\}.$

Let $z:\overline{M} \to \overline{\C}\equiv \overline{M}/\langle T_0\rangle$ denote the natural two sheeted branched covering. Up to a conformal transformation,  we will suppose that $z(P)=1/z(I(P))=0$ and $r=z(Q) \in \R-\{0\}.$ We infer that $z\circ I=\mu/\overline{z},$ and since $I$ is an involution, then $\mu \in \R-\{0\}.$ Up to the change $z \to \sqrt{|\mu|} z$, we can put $\mu^2=1.$ We distinguish two cases: $z \circ I=1/\overline{z}$ and  $z \circ I=-1/\overline{z}.$\\

{\em Case} 1. \ $z \circ I=1/\overline{z}.$

Up to biholomorphisms, $\overline{M}=\{(z,v)\in\overline{\C}^2;\; v^2=z(z-r)(r z-1)\}$ and $T_0(z,v)=(z,-v).$ As $T_0 \circ I=I \circ T_0$ and $I$ has no fixed points, we get $I(z,v)=(1/\overline{z},\overline{v}/\overline{z}^2).$ Consider $T_1\in \Sym_+(X)-\{\Id,T_0\}$ and note that $T_1(P)=I(P)$ (otherwise $T_1$ would be an holomorphic involution fixing $P$ and $I(P),$ hence $T_1=T_0$ which is absurd). Thus we get that $z \circ T_1=\lambda/z,$ and since $T_1$ leaves invariant the branch point set of $z,$  $\lambda=1.$

Let us determine $g.$ Basic Algebraic Geometry says that $g$ is a rational function of $z$ and $v.$ Moreover, we know that $g\circ I=1/\overline{g}$ and $g \circ T_0=\pm{g}$ (recall that $T_0(P)=P$ and so $(g\circ T_0)(P)=0$).  

Suppose for a moment that $g \circ T_0=g.$ In this case, $g=R(z)$ where $R(z)$ is a rational function of $z.$   
Up to rotations about the $x_3$-axis, it is easy to get 
$g=z(z-a)/(\overline{a} z-1),$ $a \in \C$. 
Here we have taken into account that $g$ has degree four, $g(0)=0$ and $g \circ I = 1/\overline{g}$.  
Then the conditions \eqref{eq:R} and $I^*(\phi_3)=\overline{\phi}_3$ imply that $\phi_3=i A (z-a)(\overline{a} z-1)(z v)^{-1} dz,$ $A \in \R-\{0\}$ (up to scaling in $\L^3$, we may assume $A\in\{\pm 1\}$).  
Furthermore, $g \circ T_1=\pm 1/g$ forces  $a \in \R.$ Let $\gamma \in H_1(\overline{M},\Z)$ denote the loop $z^{-1}([r,1/r])$ and observe that $I_*(\gamma) =\gamma,$ where $I_*:H_1(\overline{M},\Z)\to H_1(\overline{M},\Z)$ is the isomorphism induced by $I.$ By the same argument as in Lemma~\ref{lm:gphi=0},  $X'$ is well defined if and only if  
$$\int_\gamma \phi_3 g=0.$$ 
However,   $\phi_3 g= i A (z-a)^2v^{-1} dz$  has non zero integral along $[r,1/r],$ getting a contradiction.

Assume now that $g \circ T_0=-g.$ Then $g=R(z) v,$ where $R(z)$ is a rational function of $z.$ 
By reasoning as above, we get either 
$$g=\frac{v(z-a)}{(z-r)(a z-1)} 
  \quad \mbox{or} \quad
  g=\frac{v(z-a)}{(r z+1)(a z-1)},
$$ 
and in any case $\phi_3=i (z-a) (a z-1)z^{-2} dz,$ where  $a \in \R-\{0,1/r\}.$ Since $\phi_3$ has no real periods, its residue at $z=0$ must be real, that is to say, $1+a^2=0,$ a contradiction.

Therefore, this case is impossible.\\
 
{\em Case} 2. \ $z \circ I=-1/\overline{z}.$

By reasoning as above, we get $\overline{M}=\{(z,v)\in\overline{\C}^2;\; v^2=z(z+r)(r z-1)\},$ $r \in \R-\{0\},$ $I(z,v)=(-1/\overline{z},\pm\overline{v}/\overline{z}^2),$ $T_0(z,v)=(z,-v)$ and $T_1(z,v)=(-1/z,\pm v/z^2).$ 

Suppose that $g \circ T_0=g$ and  $g=R(z),$ where $R(z)$ is a rational function of $z.$ Up to a rotation about the $x_3$-axis, we get $g=z (z-a)/(a z+1),$ $\phi_3=A (z-a)(a z+1)(z v)^{-1} dz,$ $a \in \R,$ $A \in \{\pm 1,\pm i\}.$ Consider the interval $J \subset \R$ with endpoints in $\{0,-r,1/r\}$ and such that $I_*(\gamma) =\gamma,$ where $\gamma=z^{-1}(J).$ 
By reasoning as above, we get 
$$\int_\gamma \phi_3 g \neq 0,$$  
contradicting the period condition.  

Assume now that $g \circ T_0=-g.$ As above, either  
$$g=\frac{v(z+a)}{(z+r)(a z-1)} 
  \quad \mbox{or} \quad 
  g=\frac{v(z+a)}{(r z-1)(a z-1)}.
$$ 
Up to relabeling $r=z(I(Q))$, we can deal only with the first case 
$$g=\frac{v(z+a)}{(z+r)(a z-1)}.$$ 
Then $\phi_3=i (z-a) (a z+1)z^{-2} dz,$ where  $a \in \R-\{r\}.$  
Moreover, the condition $g \circ I=1/\overline{g}$ forces that $I(z,v)=(-1/\overline{z},-\overline{v}/\overline{z}^2).$ Since $\phi_3$ has no real periods, its residue at $z=0$ vanishes and $a^2=1$ (up to the changes $z \to -z$ and $r\to-r,$ we can put $a=1$). 
These Weierstrass data correspond to the examples in Theorem \ref{th:exist}, concluding the proof.  
\end{proof}

Theorem B in the introduction follows from Theorems \ref{th:exist} and \ref{th:uniq}.  


%
\bigskip
\address{ 
Department of Mathematics \\
Fukuoka University of Education \\
Munakata, Fukuoka 811-4192 \\
Japan
}
{fujimori@fukuoka-edu.ac.jp}

\address{
Departamento de Geometr\'{\i}a y Topolog\'{\i}a \\
Facultad de Ciencias \\
Universidad de Granada \\
18071 Granada \\
Spain
}
{fjlopez@ugr.es}
\end{document}